\documentclass[12pt]{article}
\usepackage{e-jc}

\usepackage{amsmath,amssymb}
\usepackage{graphicx}
\usepackage[colorlinks=true,citecolor=black,linkcolor=black,urlcolor=blue]{hyperref}

\usepackage{enumerate}
\usepackage{mathbbol}
\usepackage{braket}
\usepackage{longtable}
\usepackage{caption}
\usepackage{mathtools}
\usepackage[small]{diagrams}
\usepackage{microtype}
\usepackage{chngcntr}
\counterwithin{figure}{section}
\counterwithin{theorem}{section}

\DeclareSymbolFontAlphabet{\mathbb}{AMSb}
\DeclareSymbolFontAlphabet{\mathbbl}{bbold}

\newenvironment{claimproof}[1][Proof of Claim]{\begin{proof}[#1]}{\end{proof}}

\newtheorem*{lemma*}{Lemma}
\newtheorem*{theorem*}{Theorem}
\newtheorem*{proposition*}{Proposition}
\newtheorem*{fact*}{Fact}
\newtheorem*{notation*}{Notation}
\newtheorem*{conventions*}{Conventions}
\newtheorem*{remark*}{Remark}
\newtheorem*{corollary*}{Corollary}
\newtheorem*{conjecture*}{Conjecture}
\newtheorem*{problem*}{Problem}
\newtheorem*{question*}{Question}

\newtheorem{assumption*}{Assumption}

\theoremstyle{remark}
\newtheorem*{claim*}{Claim}

\numberwithin{equation}{section}

\newcommand{\N}{\mathbb{N}}

\newcommand{\Q}{\mathbb{Q}}

\newcommand{\bs}{\backslash}
\newcommand{\ceil}[1]{\left \lceil #1 \right \rceil }
\newcommand{\floor}[1]{\left \lfloor #1 \right \rfloor}
\newcommand{\meet}{\wedge}
\newcommand{\join}{\vee}

\renewcommand\AA{{\mathcal A}}

\newcommand\KK{{\mathcal K}}

\newcommand\bbzero{\mathbbl 0} 
\newcommand\bbone{\mathbbl 1}  

\newcommand\includedin{\subseteq}

\newcommand\union{\cup}

\newcommand\fraisse{Fra\"\i ss\'e }

\newcommand{\arr}{\overset {E}{\rightarrow}}
\newcommand{\ra}[1]{\overset {#1}{\rightarrow}}
\newcommand{\la}[1]{\overset {#1}{\leftarrow}}

\newcommand{\To}{{\Rightarrow}}
\newcommand{\From}{{\Leftarrow}}

\newarrow{Double}<--->

\newcommand{\opp}{\mathrm{opp}}
\newcommand{\red}{\mathrm{red}}

\dateline{Nov 29, 2017}{June 6, 2018}{TBD}


\Copyright{  The authors. Released under the CC BY-ND license (International 4.0).}

\title{Homogeneous 3-dimensional permutation structures}

\author{Samuel Braunfeld\\
\small Department of Mathematics\\[-0.8ex]
\small Rutgers University\\[-0.8ex] 
\small New Jersey, U.S.A.\\
\small\tt swb52@math.rutgers.edu\\}

\begin{document}

\maketitle


\begin{abstract}
We provide a classification of the homogeneous 3-dimensional permutation structures, i.e. homogeneous structures in a language of 3 linear orders, partially answering a question of Cameron [\textit{Homogeneous permutations}, Electronic Journal of Combinatorics, 2002]. We also arrive at a natural description of all known homogeneous finite-dimensional permutation structures by modifying the language used in the construction from \textit{The Lattice of Definable Equivalence Relation in Homogeneous $n$-Dimensional Permutation Structures} [Samuel Braunfeld, Electronic Journal of Combinatorics, 2016], completing the catalog begun there.
\end{abstract}

\section{Introduction}
   \newtheorem*{lemma:trianglereduce}{\bf{Lemma \ref{lemma:trianglereduce}}}

In \cite{Cameron}, Cameron classified the homogeneous permutations, which he identified with homogeneous structures consisting of two linear orders. He then posed the problem of classifying the homogeneous structures consisting of $n$ linear orders for any $n$ \cite[\S 6]{Cameron}, which we call \textit{$n$-dimensional permutation structures}. The first step toward such a classification is to produce a catalog of examples occurring ``in nature'', undertaken in \cite{Braun}, which introduced a construction for producing many new imprimitive examples. However, the construction did not quite capture all examples that were known at the time (see Examples 1,2 in \S 2 below).

While working in \cite{RamExp} on the structural Ramsey property for the structures from \cite{Braun}, it became apparent that rather than working with linear orders, the proper language decomposes linear orders that are $E$-convex, i.e. such that $E$-classes are convex with respect to the order, for some $\emptyset$-definable equivalence relation $E$ into an order within $E$-classes and an order on the quotient; we call these pieces of orders \textit{subquotient orders}.

\begin{definition}
Let $X$ be a structure, and $E \leq F$ equivalence relations on $X$. A \textit{subquotient-order from $E$ to $F$} is a partial order on $X/E$ in which two $E$-classes are comparable iff they lie in the same $F$-class (note, this pulls back to a partial order on $X$). Thus, this partial order provides a linear order of $C/E$ for each $C \in X/F$. We call $E$ the \textit{bottom relation} and $F$ the \textit{top relation} of the subquotient-order. 
\end{definition}

When the construction from \cite{Braun} is carried out with subquotient orders rather than linear orders, it produces all known examples of homogeneous finite-dimensional permutation structures. The following question asks whether the list produced by this modified construction is complete, using terminology introduced in \cite{RamExp} and reviewed in \S 2.

\begin{question} \label{que:sqoconj}
   Is every homogeneous finite-dimensional permutation structure with lattice of $\emptyset$-definable equivalence relations isomorphic to $\Lambda$ interdefinable with the \fraisse limit of some well-equipped lift of the class of all finite $\Lambda$-ultrametric spaces, for some distributive lattice $\Lambda$?
\end{question}
   
The following classification in the primitive case was conjectured in \cite{Braun}. We show in Proposition \ref{prop:sqoprim} that this would follow from a positive answer to the above question.

\begin{conjecture}[Primitivity Conjecture, \cite{Braun}]
Every primitive homogeneous finite dimensional permutation structure can be constructed by the following procedure.
\begin{enumerate}
\item Identify certain orders, up to reversal.
\item Take the \fraisse limit of the resulting amalgamation class, getting a fully generic structure, possibly in a simpler language.
\end{enumerate}
\end{conjecture} 


For example, the primitive permutations in Cameron's classification are obtained by identifying the two orders to get $(\Q, <, <)$, identifying the second order as the reversal of the first to get $(\Q, <, >)$, or making no identifications to get the \fraisse limit of all finite permutations.

We next extract a consequence of the Primitivity Conjecture. Lemma \ref{lemma:trianglereduce} proves this subconjecture for $k=3$, and it seems it should be tractable for several further small values of $k$ via the methods used there.

\begin{conjecture}
Let $\Gamma$ be a homogeneous $k$-dimensional permutation structure realizing all 3-types. Then $\Gamma$ is fully generic.
\end{conjecture}

The main result of the present paper is the following classification of the homogeneous 3-dimensional permutation structures, which gives a positive answer to Question \ref{que:sqoconj} in this case.

\begin{definition}
Given structures $\Gamma_1, \Gamma_2$, the \textit{composition of $\Gamma_1$ with $\Gamma_2$}, denoted  $\Gamma_1[\Gamma_2]$, is the structure obtained by expanding $\Gamma_1$ with an equivalence relation $E$, and replacing the points of $\Gamma_1$ by $E$-classes that are copies of $\Gamma_2$.
\end{definition}

\begin{theorem}[The Catalog] \label{thm:catalog}
Let $(\Gamma, <_1, <_2, <_3)$ be a homogeneous 3-dimensional permutation structure. We use $\Gamma^{(g)}_i$ to denote the generic $i$-dimensional permutation structure; in particular $\Gamma^{(g)}_0$ is a set equipped only with equality. Then $\Gamma$ is quantifier-free interdefinable with one of the following 16 structures.

\begin{enumerate}
\item $\Gamma$ has no non-trivial $\emptyset$-definable congruence
	\begin{enumerate}
	\item $\Gamma$ is primitive: $\Gamma = \Gamma^{(g)}_1, \Gamma^{(g)}_2,$ or $\Gamma^{(g)}_3$.
	\item $\Gamma$ is imprimitive: $\Gamma$ is the expansion  by a generic linear order of $\Gamma^{(g)}_1[\Gamma^{(g)}_j]$, for $j \in \set{0, 1}$.
	\end{enumerate}
\item $\Gamma$ has a non-trivial $\emptyset$-definable congruence
	\begin{enumerate}
	\item $\Gamma$ is a repeated composition of primitive structures: For any multisubset $I \subset \set{1,2}$ such that $|I|>1$ and $\sum_{i \in I} 2^i \leq 8$, $\Gamma$ is the composition in any order of $\Gamma^{(g)}_i$ for $i \in I$.
	\item $\Gamma$ is a composition of primitive and imprimitive structures: Let $\Gamma^*$ denote the structure from $(1b)$ with $j=0$. Then $\Gamma = \Gamma^*[\Gamma^{(g)}_1]$ or $\Gamma^{(g)}_1[\Gamma^*]$.
	\end{enumerate}
\end{enumerate}

\end{theorem}

The classification proceeds in two stages. First, we confirm the Primitivity Conjecture for 3 orders using explicit amalgamation arguments. Then for the imprimitive case, we pick a minimal non-trivial equivalence relation $E$. The Primitivity Conjecture makes it fairly clear what happens on $E$-classes, and some analysis of the type structure between $E$-classes eventually allows us to carry out an inductive classification.

\begin{corollary} \label{cor:confirmsqo}
   Every homogeneous 3-dimensional permutation structure is interdefinable with the \fraisse limit of some well-equipped lift of the class of all finite $\Lambda$-ultrametric spaces, for some distributive lattice $\Lambda$.
\end{corollary}

Despite the fact that assuming a positive answer to Question \ref{que:sqoconj} gives a simple description of \textit{all} finite-dimensional permutation structures, it is difficult to determine the corresponding catalog for a \textit{fixed number} of linear orders. This is because it is not known what lattices of $\emptyset$-definable equivalence relations can be realized with a given number of orders (this problem is discussed, and an upper bound provided, in \cite[\S 3.4]{Braun}), nor is it true that one needs at most $n$ orders to represent a structure with at most $2^n$ 2-types.

\begin{question} \label{qu:represent}
Given a lattice $\Lambda$, what is the minimal $n$ such that $\Lambda$ is isomorphic to the lattice of $\emptyset$-definable equivalence relations of some homogeneous $n$-dimensional permutation structure?

Given a homogeneous finite-dimensional permutation structure $\Gamma$ presented in a language of equivalence relations and subquotient orders, what is the minimal $n$ such that $\Gamma$ is quantifier-free interdefinable with an $n$-dimensional permutation structure? 
\end{question}

Thus, Corollary \ref{cor:confirmsqo} is not proven by first producing a conjectural classification and then confirming it. Rather, it is proven by observing that all the structures appearing in the classification may be presented appropriately.

Finally, although we have a positive answer to Question \ref{que:sqoconj} in the case of 3 orders, a plausible exceptional imprimitive structure arises in the analysis (see Lemma \ref{lemma:dense}) that is ultimately shown not to exist. However, the proof of non-existence makes use of the limited type structure with 3 orders, and it seems possible similar structures will appear in the richer languages afforded by more orders. 

\section{$\Lambda$-Ultrametric Spaces and Subquotient Orders}
This section is not strictly needed for the classification of the 3-dimensional case, but does provide context by giving the necessary background for Question \ref{que:sqoconj}. Because we hewed to the language of linear orders, we were unable to provide a satisfactory catalog in \cite{Braun} of homogeneous finite-dimensional permutation structures, since some known examples were not produced by the construction. After modifying the construction to work with subquotient orders, we show such examples, for which we can now give a straightforward description. The notion of a \textit{well-equipped lift}, which ensures that we may translate from the language of subquotient orders to linear orders, is then introduced, thus defining all the terms in Question \ref{que:sqoconj}, which in turn provides a conclusion to our catalog. Finally, we show that Question \ref{que:sqoconj} subsumes the Primitivity Conjecture from \cite{Braun}. 

\begin{definition} Let $\Lambda$ be a complete lattice. A \textit{$\Lambda$-ultrametric space} is a metric space where the metric takes values in $\Lambda$ and the triangle inequality involves join rather than addition, i.e. $d(x, z) \leq d(x, y) \join d(y, z)$.
\end{definition}

The following theorem shows that $\Lambda$-ultrametric spaces are equivalent to structures equipped with a lattice of equivalence relations isomorphic to $\Lambda$, or to substructures of such structures. While the lattice of equivalence relations may collapse when passing to a substructure, such as a single point, $\Lambda$-ultrametric spaces have the benefit of keeping $\Lambda$ fixed under substructures.

\begin{theorem}[\cite{Braun}]\label{theorem:isomorphism}
For a given finite lattice $\Lambda$, there is an isomorphism between the category of $\Lambda$-ultrametric spaces and the category of structures consisting of a set equipped with a family of equivalence relations, closed under taking intersections in the lattice of all equivalence relations on the set, and labeled by the elements of $\Lambda$ in such a way that the map from $\Lambda$ to the lattice of equivalence relations is meet-preserving. Furthermore, the functors of this isomorphism preserve homogeneity.
\end{theorem}

Although we do not prove this theorem here, we will define the functors giving this isomorphism. Given a system of equivalence relations as specified above, we obtain the corresponding $\Lambda$-ultrametric space by taking the same universe and defining $d(x, y) = \bigwedge \set{\lambda \in \Lambda | x E_\lambda y}$. In the reverse direction, given a $\Lambda$-ultrametric space, we obtain the corresponding structure of equivalence relations by taking the same universe and defining $E_\lambda = \set{(x,y) | d(x,y) \leq \lambda}$.

Since the lattices we are considering will always be finite, they will have a top and bottom element, denoted $\bbone$ and $\bbzero$, respectively. Thus, $d(x, y) = \bbzero$ iff $x = y$.

For every finite \textit{distributive} lattice $\Lambda$, a construction was given in \cite{Braun} producing a countable homogeneous finite-dimensional permutation structure $\Gamma$, such that the lattice of $\emptyset$-definable equivalence relations in $\Gamma$ is isomorphic to $\Lambda$. The structure $\Gamma$ is naturally presented as a $\Lambda$-ultrametric space, equipped with multiple orders. When $\Lambda$ is distributive, the class of all finite $\Lambda$-ultrametric spaces is an amalgamation class. The structure $\Gamma$ is constructed by taking the generic $\Lambda$-ultrametric space, and adding linear orders that are generic, except that they are required to be convex with respect to a prescribed set of equivalence relations corresponding to a chain of meet-irreducibles in $\Lambda$; enough such linear orders have to be added so that every meet-irreducible is convex with respect to at least one order, and there are further complications if $\bbzero$ (equality) is meet-reducible.

Working at the level of subquotient orders requires a straightforward revision of the proof of amalgamation in \cite[Lemma 3.7]{Braun}. The proof is actually simplified by the language change, eliminating a special case the construction required when $\bbzero$ is meet-reducible. This yields the following theorem.

\begin{theorem} \label{theorem:amalg}
Let $\Lambda$ be a finite distributive lattice, and $\Gamma$ the generic $\Lambda$-ultrametric space. For each meet-irreducible $E \in \Lambda$, fix a function $f_E: \set{F \in \Lambda| E<F} \to \N$. Then there is a homogeneous expansion of $\Gamma$, which is generic in a natural sense, adding, for each meet-irreducible $E \in \Lambda$ and $F>E$, $f_E(F)$ subquotient orders from $E$ to $F$.

To be more precise, the following holds. Let $\AA^*$ be the class of finite structures $(A,d,\set{<_{E_i, j}}_{j=1}^{n_i})$ satisfying the following conditions.
\begin{itemize}
\item $(A,d)$ is  a $\Lambda$-ultrametric space.
\item $<_{E_i, j}$ is a subquotient order with bottom relation $E_i$, for some meet-irreducible $E_i \in \Lambda$, and top relation $F_{i, j} \in \Lambda$.
\end{itemize}
Then $\AA^*$ is an amalgamation class, and its \fraisse limit is the desired expansion if $\set{E_i}$ and $\set{F_{i,j}}$ are chosen to match $f_E$ from above.
\end{theorem}
\begin{proof}
See Appendix A.
\end{proof}

We now define two useful constructions with subquotient orders, and then give two examples of homogeneous finite-dimensional permutation structures not produced by the construction of \cite{Braun}, but which can be produced once linear orders are replaced by subquotient orders.

\begin{definition}
If $x$ is an $E$-class, and $F$ an equivalence relation above $E$, then $x/F$ will represent the $F$-class containing $x$.
\end{definition}

\begin{definition}
Let $<_{E, F}$ be a subquotient order with bottom relation $E$ and top relation $F$, and let  $<_{F, G}$ be a subquotient order with bottom relation $F$ and top relation $G$. Then the \textit{composition} of $<_{F, G}$ with $<_{E, F}$, denoted $<_{F, G}[<_{E, F}]$, is the subquotient order with bottom relation $E$ and top relation $G$ given by $x \mathrel{<_{F, G}[<_{E, F}]} y$ iff either of the following holds.
\begin{enumerate}
\item $x$ and $y$ are in the same $F$-class, and $x <_{E, F} y$
\item $x$ and $y$ are in distinct $F$-classes, and $x/F <_{F, G} y/F$.
\end{enumerate}
\end{definition}

\begin{definition}
Let $<_{E, F}$ be a subquotient order with bottom relation $E$ and top relation $F$, and let $G$ be an equivalence relation lying between $E$ and $F$. Then the \textit{restriction of $<_{E, F}$ to $G$}, denoted $<_{E, F} \upharpoonright_G$, is the subquotient order with bottom relation $E$ and top relation $G$ given by $x <_{E, F} \upharpoonright_G y$ iff $x$ and $y$ are in the same $G$-class and $x <_{E, F} y$.
\end{definition}

\begin{example} \label{ex:complex3}
Let $\AA$ be the amalgamation class consisting of all finite structures in the language $\set{E, <_1, <_2}$, where $E$ is an equivalence relation, $<_1$ is a linear order, and $<_2$ is an $E$-convex linear order that agrees with $<_1$ on $E$-classes. 

 Let $\AA'$ be the class of all finite structures in the language $\set{E', <'_1, <'_2}$, where $E'$ is an equivalence relation, $<'_1$ is a subquotient order from $=$ to $\bbone$, and $<'_2$ a subquotient order from $E'$ to $\bbone$. This is also an amalgamation class, and its Fra\"\i ss\'e limit $\Gamma'$ is interdefinable with the Fra\"\i ss\'e limit $\Gamma$ of $\AA$. 

To define $\Gamma$ from $\Gamma'$, let ${<_1} = {<'_1}$, and let ${<_2} = {<'_2[<'_1 \upharpoonright_E]}$. To define $\Gamma'$ from $\Gamma$, let ${<'_1} = {<_1}$, and let $x <'_2 y$ iff $\neg xEy$ and $x <_2 y$.

We note that in Theorem \ref{thm:catalog}, this is the structure in $(1b)$ with $j=0$.
\end{example}

\begin{example} \label{ex:fullQ2}
For a more complex example of the use of subquotient orders, consider the full product $\Q^2$. This is a homogeneous structure with universe $\Q^2$ in the language $\set{E_1, E_2, <_1, <_2}$, where $E_1$ and $E_2$ are the relations defined by agreement in the first and second coordinates, respectively, $<_1$ is a generic subquotient order from $E_1$ to $\bbone$, and $<_2$ is a generic subquotient order from $E_2$ to $\bbone$.

Since $E_1 \meet E_2 = \bbzero$, we see that $<_1$ defines a linear ordering on each $E_2$-class, and $<_2$ defines a linear ordering on each $E_1$-class. Thus, the composition (abusing notation slightly) $<_1[<_2]$ defines an $E_1$-convex linear order, and $<_2[<_1]$ defines an $E_2$-convex linear order. 

As this structure requires four linear orders, it does not appear in our catalog.
\end{example}

It is not clear how either of these examples can be produced by a generic construction using linear orders. Neither can be obtained by the construction from \cite{Braun}. There the only constraints we put on the linear orders were convexity conditions, which involves forbidding substructures of order 3. However, in Example 1, we must forbid a substructure of order 2 to force $<_1$ and $<_2$ to agree between $E$-related points. In Example 2, we must forbid the following substructure of order 4 (as well as another symmetric substructure):
\begin{enumerate}
\item $x_1 E_1 x_2$, $y_1 E_1 y_2$, $\neg x_1 E_1 y_1$
\item $x_1 E_2 y_1$, $x_2 E_2 y_2$, $\neg x_1 E_2 x_2$ 
\item $x_1 <_1 x_2$, $y_2 <_1 y_1$
\end{enumerate}

\begin{definition}
Let $\Lambda$ be a finite distributive lattice, and let $L$ be a language consisting of relations for the distances in $\Lambda$ and finitely many subquotient orders, labeled with their top and bottom relations. We say that the language $L$ is \textit{$\Lambda$-well-equipped} if $E \in \Lambda$ appears as the bottom relation of some subquotient order in $L$ with distinct bottom and top relations iff $E$ is meet-irreducible, for every $E \in \Lambda$. 

If $\AA_\Lambda$ is the class of all finite $\Lambda$-ultrametric spaces, and $L$ a $\Lambda$-well-equipped language, we will call $\vec \AA_\Lambda$ a \textit{well-equipped lift} of $\AA_\Lambda$ if it consists of all finite $\Lambda$-ultrametric spaces equipped with subquotient orders from $L$.
\end{definition}

\begin{proposition}[\cite{RamExp}, Proposition 3.12]
Let $\Lambda$ be a finite distributive lattice, $\AA_\Lambda$ be the class of all finite $\Lambda$-ultrametric spaces, and $\vec \AA_\Lambda$ a well-equipped lift of $\AA_\Lambda$, with Fra\"\i ss\'e limit $\vec \Gamma$. Then the relations of $\vec \Gamma$ are interdefinable with a set of linear orders.
\end{proposition}

We close this section by showing the Primitivity Conjecture follows from a positive answer to Question \ref{que:sqoconj}.
\begin{proposition} \label{prop:sqoprim}
Let $\Gamma$ be the generic $n$-dimensional permutation structure, in the language $\set{<_1, ..., <_n}$, and let $<$ be a definable linear order on $\Gamma$. Then there is an $i \in [n]$ such that ${<} = {<_i}$ or ${<} = {<_i^{\opp}}$.  
\end{proposition}
\begin{proof}
Note that $<$ must be a union of 2-types $\cup q_i$, and for each 2-type, exactly one of it and its opposite must be appear as some $q_i$. We may assume $q_0 = \set{x <_1 y, ..., x <_n y}$. If the conclusion is false, then for each $i \in [n]$, there must be a type $p_i = q_j$ for some $j$, such that $p_i \vdash y <_i x$. 

Now consider the partial structure on $\set{x_1, ..., x_{n+1}}$ given by setting $p_i(x_i, x_{i+1})$ for each $i \in [n]$. For each $i \in [n]$, looking at $<_i$ gives a directed acyclic graph, whose transitive closure is a partial order in which $x_1$ and $x_n$ are $<_i$-incomparable. This can then be completed to a linear order in which $x_n <_i x_1$.
 
 Each $<_i$ is a linear order in the resulting structure, which is thus a substructure of $\Gamma$. However, we have $x_1 < ... < x_n$ but $x_n < x_1$. Thus $<$ is not transitive on this structure, and so does not define a linear order on $\Gamma$.
\end{proof}

\section{The Primitive Case}
In this section, we classify the primitive homogeneous 3-dimensional permutation structures, obtaining the following.

\begin{theorem}\label{theorem:primclass}
The primitive homogeneous 3-dimensional permutation structures are as predicted by the Primitivity Conjecture.
\end{theorem}

The main lemmas needed for the proof are below.

\begin{lemma}\label{lemma:trianglereduce}
Suppose $(\Gamma; <_1, <_2, <_3)$ is homogeneous and contains all 3-types. Then $\Gamma$ is generic.
\end{lemma}

\begin{lemma} \label{lemma:primtarget}
Let $\Gamma$ be a primitive homogeneous 3-dimensional permutation structure. Then all 3-types involving realized 2-types are realized.
\end{lemma}

\begin{proposition} [\cite{Braun}, {Proposition 5.1}] \label{prop:2prim}
Let $\KK$ be an amalgamation class of $n$-dimensional permutation structures. If all 3-types involving realized 2-types are realized, then the forbidden 2-types specify that certain orders agree up to reversal.
\end{proposition}

\begin{proof}[Proof of Theorem \ref{theorem:primclass}] 
By Lemma \ref{lemma:primtarget} all 3-types involving realized 2-types are realized. Thus, if no 2-types are forbidden, all 3-types are realized, and so by Lemma \ref{lemma:trianglereduce}, the resulting structure is generic. If some 2-types are forbidden, then by Proposition \ref{prop:2prim}, some orders agree up to reversal. Thus the resulting structure is quantifier-free interdefinable with a primitive homogeneous 2-dimensional permutation structure. By the classification in \cite{Cameron}, these satisfy the Primitivity Conjecture.
\end{proof}

The amalgamation diagram appearing in the following definition is the key to the proof of Proposition \ref{prop:2prim}, and will be used elsewhere in this section.

\begin{definition}
Give 2-types $p,q,r$, the \textit{$(p,q,r)$-majority diagram} is the following amalgamation diagram, where $x_1 \ra{q} x_3$ holds (and follows from $x_1 \ra{q} x_2 \ra{q} x_3$), but is not drawn.

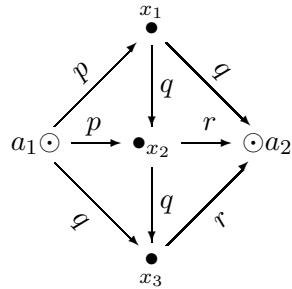
\begin{figure}[h]
\begin{diagram}
& & \overset{x_1}{\bullet} & & \\
& \ruTo^{p} & \dTo_q  & \rdTo^{q} & \\
a_1 \odot &\rTo^p & \bullet_{x_2}  & \rTo^r & \odot a_2 \\
& \rdTo_{q} & \dTo_q & \ruTo_{r} & \\
& & \underset{x_3}{\bullet} & & \\
\end{diagram}
\caption{}
\end{figure}
\end{definition}

\begin{remark}
We offer some guidance on interpreting the above amalgamation diagram. Solid points lie in the base, while circled points lie outside the base; those on the left are in the first factor, while those on the right are in the second. Labeled arrows indicate 2-types. The diagram is completed by determining $tp(a_1, a_2)$.  
\end{remark}

The following was proven in \cite[Proposition 5.1]{Braun}.

\begin{lemma}
There is a unique solution to the $(p,q,r)$-majority diagram, given by $a_1 <_i a_2$ iff $<_i$ is true in the majority of $p,q,$ and $r$.
\end{lemma}

\subsection{Reduction to 3-types} 
The following lemma strengthens an argument appearing in the proof of \cite[Theorem 1]{Cameron}.
\begin{lemma}\label{lemma:4gen}
Let $\Gamma$ be a homogeneous $k$-dimensional permutation structure that contains all configurations on $n-1$ points, where $n$ satisfies
$$\frac{n!}{(n-\ell)!} > 2^\ell k \text{ for } \ell = \floor{n/2}$$
Then $\Gamma$ is generic.

More precisely, any configuration on $N \geq n$ points is contained in the unique amalgam of $(N-1)$-point configurations.
\end{lemma}
\begin{proof}
Let $A$ be a structure on $N$ points. Let a \textit{pairing} be an $\ell$-set of unordered pairs of points from $A$, with each point appearing in at most one pair. A pairing is \textit{separated} if, for every $i \leq k$, there is a pair $\set{a_i, a'_i}$ such that $a_i$ and $a'_i$ are not $<_i$-adjacent; otherwise the pairing is unseparated.

\begin{claim*}
 There is at least one separated pairing on $A$.
\end{claim*} 

\begin{claimproof}The number of pairings is given by $\binom{n}{2\ell} \binom{2\ell}{2_1, 2_2, ..., 2_\ell}/\ell! = \frac{n!}{2^\ell \ell! (n-2\ell)!}$. We will now show the number of unseparated pairings is at most $k \binom{n-\ell}{\ell}$. First suppose $k=1$. If $N$ is even, there is only 1 unseparated pairing. If $N$ is odd, the pairing is determined after choosing any one of the odd-indexed points to not appear, so there are $\ceil{n/2}$. In both cases, there are $\binom{n-\ell}{\ell}$. For larger $k$, note that an unseparated pairing must be unseparated with respect to at least one order, so there are at most $k\binom{n-\ell}{\ell}$. By inequality in the hypothesis, we are done.
\end{claimproof}

Let $P$ be a separated pairing. By extending $A$ by a single point, we may, in every order, make one pair non-adjacent. Thus, after extending $A$ by at most $\ell-1$ points, an extension we will denote by $A^*$, we may assume that every pair in $P$ is non-adjacent in every order.

Let $(a_1, a'_1)$ be a pair from $P$, and let $F_1 = A^* \bs{\set{a_1}}, F'_1 = A^* \bs{\set{a'_1}}$, and $B_1 = A^* \bs{\set{a_1, a'_1}}$. By assumption, for every $i \leq k$, there is a point $b_i \in B_1$ that is $<_i$-between $a_1$ and $a'_1$. Thus $A^*$ is the unique amalgam of $F_1$ and $F'_1$ over $B$.

We may recursively continue this process on each factor until we have gone through all the pairs in $P$. At the end, each factor will look like a copy of $A^*$ with $\ell$ points removed, and so have size $N-1$.

\end{proof}

\begin{lemma:trianglereduce}
Suppose $(\Gamma; <_1, <_2, <_3)$ is homogeneous and contains all 3-types. Then $\Gamma$ is generic.
\end{lemma:trianglereduce}
\begin{proof}
By Lemma \ref{lemma:4gen}, if $\Gamma$ contains all 4-point configurations, it is generic.

Let $A = (\set{a,b,c,d}; <_1, <_2, <_3)$ be a substructure of $\Gamma$. There are three possible pairings: $P_1 = \set{\set{a,b}, \set{c,d}}, P_2 = \set{\set{a,c}, \set{b,d}}, P_3 = \set{\set{a,d}, \set{b,c}}$. Each order can be unseparated in at most one pairing, so if all the pairings are unseparated, each must be so with respect to a different order. By possibly relabeling the points, we may assume that $a <_1 b <_1 c <_1 d$, and by relabeling orders we may assume that $P_i$ is unseparated with respect to $<_i$.

Thus, we have that $a,c$ and $b,d$ must be $<_2$-adjacent, and $a,d$ and $b,c$ must be $<_3$-adjacent.

We may extend $A$ by a single point, $e$, that lies between $a$ and $b$ with respect to $<_1$, lies between $a$ and $c$ with respect to $<_2$, and lies between $b$ and $c$ with respect to $<_3$, and label the resulting structure $A^*$. Then, viewing $B = \set{e, c, d}$ as the base of an amalgamation digram with $F = B \cup \set{a}$ the first factor and $F' = B \cup \set{b}$ the second, we have that $A^*$ is the unique amalgam.

We now show that $F$ and $F'$ have separable pairings, and so are contained in the unique amalgam of certain 3-types. For $F$, $P = \set{\set{a,c}, \set{e,d}}$ is separated, since $e$ and $d$ are never $<_2$-adjacent and only $<_3$-adjacent if $b$ and $d$ were $<_3$-adjacent, in which case $a$ and $c$ not $<_3$-adjacent. For $F'$, $P' = \set{\set{b,c}, \set{e,d}}$ is separated, since $e$ and $d$ are never $<_3$-adjacent and only $<_2$-adjacent if $a$ and $d$ were $<_2$-adjacent, in which case $b$ and $c$ are not $<_2$-adjacent.
\end{proof}

\subsection{Notation}
There are 8 2-types, which we may associate with the vertices of the unit cube $\set{\pm 1}^3$ based on whether $<_i$ holds in the 2-type. The unit cube is bipartite, with one part consisting of the following four types at Hamming distance 2, while the other part consists of their opposites.

$$0: \la{123} \hspace{1 cm} 1: \ra{23} \hspace{1 cm} 2: \ra{13} \hspace{1 cm} 3: \ra{12}$$

We now introduce notation for 3 families of 3-types that will recur in our analysis. From left to right, the 3-types below will be denoted $p \To q$, $p \From q$, and $C_3(p,q,r)$.

\begin{figure}[h]
\begin{diagram}
 \bullet & & \rTo^{q} & & \bullet  &  \bullet & & \rTo^{q} & & \bullet & \bullet & & \rTo{q} & & \bullet  \\
& \luTo_{p} &  & \ruTo_{p}  & & &  \rdTo_p &  & \ldTo_p & & &  \luTo_p &  & \ldTo_r & \\
& & \underset{}{\bullet} & & & & &  \underset{}{\bullet} & & & & & \underset{}{\bullet} & & \\
\end{diagram}
\caption{}
\end{figure}

\subsection{Amalgamation Lemmas}

In the following lemmas, $(p,q,r,s)$ is taken to be some permutation of the types $(0,1,2,3)$.

\begin{lemma} \label{lemma:1}
Suppose $p \To q$ is forbidden. Then one of each of the following combinations of types is forbidden.
\begin{enumerate}[(A)]
\item $(p \To r$ \text{and} $C_3(p,r,s))$ or $(r \From q$ \text{and} $C_3(p,q,r))$
\item $p \From q$ or $q \From p$
\end{enumerate}
\end{lemma}
\begin{proof}
For $(A)$, we amalgamate one 3-type from each pair over an edge of type $r$. In the below diagrams, we assume $p \To r$ is realized; the arguments assuming $C_3(p,r,s)$ is realized are similar.

\begin{figure}[h]
\begin{diagram}
& & {\bullet} & & & & & & {\bullet} & & \\
& \ruTo^{p} &  & \luTo^{r} & & & & \ruTo^{p} &  & \rdTo^{p} & \\
x\odot & & \uTo{r} & & \odot y  & & x\odot & & \uTo{r} & & \odot y \\
& \rdTo_{p} &  & \ldTo_{q} & & & & \rdTo_{p} &  & \ldTo_{q} & \\
& & {\bullet} & & & & & & {\bullet} & &\\
\end{diagram}
\caption{}
\label{fig:11}
\end{figure}

We wish to argue that the only way to complete both diagrams is to take $tp(x, y) = p$. This is clear for the right diagram, by transitivity. For the left diagram, note that since $p$ and $r$ are at Hamming distance 2, $p$ and $r^{\opp}$ agree on exactly 2 orders, as do $p$ and $q^{\opp}$. Thus, by transitivity, $tp(x, y)$ must agree with $p$ in all 3 orders.

For $(B)$, we use the $(p^{\opp}, q, p)$-majority diagram (see Figure 3.1), and then take Lemma \ref{lemma:5} into account.
\end{proof}

\begin{lemma} \label{lemma:4}
Suppose $p \To q$, $C_3(p,q,r)$, and $C_3(p,q,s)$ are forbidden. If $p$ and $q$ are realized, then $q \From p$ is realized.
\end{lemma}
\begin{proof}
We try to complete the following amalgamation diagram.

\begin{figure}[h]
\begin{diagram}
 x \odot & &  & & \odot y  \\
& \rdTo_{p} &  & \ruTo_{q} & \\
& & {\bullet} & & \\
\end{diagram}
\caption{}
\label{fig:12}
\end{figure}

By assumption $tp(x, y) \neq p, r^{\opp}, s^{\opp}$. The remaining types, except $q$, are ruled out by transitivity.
\end{proof}

\begin{lemma} \label{lemma:5}
Suppose $p \To q$ is forbidden. If $p$ and $q$ are realized, then $q \To p$ is realized.
\end{lemma}
\begin{proof}
We try to complete the amalgamation diagram in Figure \ref{fig:13}.

\begin{figure}[h]
\begin{diagram}
 x \odot & &  & & \odot y  \\
& \rdTo_{p} &  & \ldTo_{q} & \\
& & {\bullet} & & \\
\end{diagram}
\caption{}
\label{fig:13}
\end{figure}

By assumption $tp(x, y) \neq p$.  The remaining types, except $q^{\opp}$, are ruled out by transitivity.
\end{proof}

\subsection{Case Division}
The proof of Lemma \ref{lemma:primtarget} proceeds by consideration of several cases. However, the following lemma provides a uniform point of departure.

\begin{lemma}
If $\Gamma$ is primitive and omits a 3-type then without loss of generality it omits the 3-type of type $(0 \Rightarrow 1)$ while realizing the 2-types $0,1$.
\end{lemma}
\begin{proof}
We may assume that at least 3 of the 2-types $0,1,2,3$, say $0,1,2$ after relabeling, are realized, since otherwise we reduce to the case of fewer linear orders. 

If any 3-type $p \Rightarrow q$ or $p \Leftarrow q$ is forbidden while $p$ and $q$ are realized, then by reversing the orders and changing the language we may assume that $0 \Rightarrow 1$ is forbidden. So assume this is not the case. 

By the above paragraphs, we may construct the standard $(0,1,2)$-majority diagram, which shows $3$ is realized as well. Up to a change of language, the forbidden 3-type must be of the form $C_3(0,1,2)$. But this is a substructure of the unique solution to the $(1^{\opp}, 0^{\opp}, 2^{\opp})$-majority diagram.
\end{proof}

\begin{remark}
Although we may assume the 2-types $1$ is realized, we may not want to, since it breaks the symmetry between $1,2,$ and $3$. Thus, this will not be assumed unless otherwise noted.
\end{remark}

We now divide into cases the ways Lemma \ref{lemma:primtarget} might fail.
\begin{list}{}{}
\item \textbf{Case 1:} All 3-types of type $0 \To p$ ($p=1,2,3)$ are forbidden, and $0$ is realized.
\item \textbf{Case 2:} For a given pair of 2-types $p,q$ at Hamming distance 2, at most 2 3-types of type $p \To q$ are forbidden, and $0,1$ are realized.
\begin{list}{}{}
\item \textbf{Case 2.1:} There exist $p,q,r$ at Hamming distance 2 such that $p \To q$ and $p \To r$ are forbidden.
\item \textbf{Case 2.2:} For any $p,q,r$ at Hamming distance 2, at most one of $p \To q$ and $p \To r$ is forbidden.
\end{list}  
\end{list}

We also wish to divide Case 1 into subcases, so assume $0 \To p$ is forbidden, for $p = 1,2,3$. Consider the directed graph with vertex set $\set{1,2,3}$, and an edge $(p,q)$ when the type $C_3(0, p, q)$ is forbidden.

By Lemma \ref{lemma:1}, for any arrangement $(p,q,r)$ of the vertices, either $(p,q)$ or $(q,r)$ is an edge. Thus, $D$ contains a symmetric edge $p \leftrightarrow q$, which we may assume is $1 \leftrightarrow 2$, and $D$ has at least 4 edges.
We now subdivide Case 1 as follows.

\begin{list}{}{}
\item \textbf{Case 1.1:} $D$ has 6 edges.
\item \textbf{Case 1.2:} $D$ has 5 edges.
\item \textbf{Case 1.3:} $D$ has 4 edges.
\end{list}

\subsection{Proof of Lemma \ref{lemma:primtarget}}
The proof proceeds by starting with the assumptions of one of the subcases and then repeatedly applying the amaglamation lemmas \ref{lemma:1}, \ref{lemma:4}, and \ref{lemma:5} until reaching a contradiction. This contradiction could either be that a structure is both forbidden and realized, or could be a violation of the primitivity constraint by the appearance of a definable equivalence relation. 

More explicitly, the 2-types $p_1, ..., p_k$ generate a definable equivalence relation if every 3-type on points $x,y,z$ satisfying the following is forbidden.
\begin{enumerate}
\item $tp(x, y), tp(y, z) \in \set{p_1, ..., p_k}$
\item $tp(x, z) \not\in \set{p_1, ..., p_k}$
\end{enumerate} 

The proofs are presented in tables. In each line, some 3-type is shown to be realized or forbidden. The reason is given; if the reason is one of the amalgamation lemmas then the assignment of $(p,q,r,s)$ is given; finally the previous lines used are given. When one of the amalgamation lemmas is used with opposite types, so for example $p \From q$ is assumed forbidden rather than $p \To q$, an ``R'' (for ``reversed'') is appended to the name of the lemma.

In the tables, we assume all 2-types are realized; after every table is a remark noting the alterations required if some 2-type is forbidden.

\begin{longtable}{|c|c|c|c|c|c|}
\caption{Case 1.1} \\
\hline 
\multicolumn{1}{|c|}{Line} &
 \multicolumn{1}{c|}{Realized} & 
 \multicolumn{1}{c|}{Forbidden} &
\multicolumn{1}{c|}{Reason} &
 \multicolumn{1}{c|}{(p,q,r,s)} & 
 \multicolumn{1}{c|}{Used} \\ \hline 
\endfirsthead
   1. & & $0 \To p$ & Case 1 & & \\
   2. & & $C_3(0,p,q)$ & Case 1.1 & & \\
   3. & $p \From 0$ & & \ref{lemma:4} & & 1,2 \\
   4. & & $0 \From p$ & \ref{lemma:1}B & &1,3 \\
   \hline
      \multicolumn{6}{|l|}{{Now the 2-type $0$ generates an equivalence relation, contradicting primitivity.}} \\ \hline
\end{longtable}

\begin{remark}
This proof works with some 2-type forbidden. If $1,2$, or $3$ is forbidden, then the corresponding case of line 4 follows without needing line 3.
\end{remark}

The treatment of the remaining cases follows the same scheme at somewhat greater length, and shows that the amalgamation lemmas given previously suffice to complete the analysis. Other methods will be required in \S 4.

\begin{longtable}{|c|c|c|c|c|c|} 
\caption{Case 1.2} \\
\hline 
\multicolumn{1}{|c|}{Line} &
 \multicolumn{1}{c|}{Realized} & 
 \multicolumn{1}{c|}{Forbidden} &
\multicolumn{1}{c|}{Reason} &
 \multicolumn{1}{c|}{(p,q,r,s)} & 
 \multicolumn{1}{c|}{Used} \\ \hline 
\endfirsthead
   1. & & $0 \To p$ & Case 1 & & \\
   2. &  &$C_3(0, p, q)$    & Case 1.2 & & \\
      &  &except $C_3(0,3,2)$   &  & & \\
   3. &$C_3(0, 3, 2)$  &  & Case 1.2 & & \\
   4. &$1 \From 0$  &   & \ref{lemma:4} &(0,1,2,3) & 1,2 \\
   5. &$2 \From 0$  &   & \ref{lemma:4} &(0,2,1,3) & 1,2 \\
   6. & &$0 \From 1$   & \ref{lemma:1}B &(0,1,2,3) & 1,4 \\
   7. & &$0 \From 2$   & \ref{lemma:1}B &(0,2,1,3) & 1,5 \\
   8. &$0 \From 3$ &   & Primitivity & & 1,6,7 \\
   9. & &$3 \From 0$   & \ref{lemma:1}B &(0,3,1,2) & 1,8 \\
   10. & &$3 \From 2$   & \ref{lemma:1}AR &(3,0,2,1) & 5,9 \\
   11. & &$3 \To 2$  or $2 \To 3$  & \ref{lemma:1}BR &(3,2,0,1) & 10 \\
   12. & $2 \To 0$ & & \ref{lemma:5} & (0,2,1,3) & 1 \\
   13. & $3 \To 0$ & & \ref{lemma:5} & (0,3,1,2) & 1 \\
   14. & &$3 \To 0$ or $2 \To 0$   & \ref{lemma:1}A &(3,2,0,1) & 3,8,11 \\
   & & & &or (2,3,0,1) & \\
   \hline 
   \multicolumn{6}{|l|}{{Now line 14 contradicts lines 12 and 13.}} \\ \hline
\end{longtable}

\begin{remark}
This proof works with some 2-types forbidden. By assumption, the types $2$ and $3$ are realized. The assumption that the type $1$ is realized only appears in line $4$, which becomes unnecessary if $1$ is forbidden since line 4 is only used for line $6$.
\end{remark} 

Case 1.3 requires further subdivision according to our assumptions on the directed graph $D$. We draw the $D$ corresponding to each of the further subcases.

\begin{figure}[h]
\begin{diagram}
1 \bullet & & \rDouble^{1.3.1} & & \bullet 2 & 1 \bullet & & \rDouble^{1.3.2} & & \bullet 2 & 1 \bullet & & \rDouble^{1.3.3} & & \bullet 2 \\
& \luDouble &  &  & & &  \rdTo &  & \ldTo & & &  \luTo &  & \ruTo & \\
& & \underset{3}{\bullet} & & & & &  \underset{3}{\bullet} & & & & & \underset{3}{\bullet} & & \\
\end{diagram}
\caption{}
\label{fig:14}
\end{figure}

\begin{longtable}{|c|c|c|c|c|c|} 
\caption{Case 1.3.1} \\
\hline 
\multicolumn{1}{|c|}{Line} &
 \multicolumn{1}{c|}{Realized} & 
 \multicolumn{1}{c|}{Forbidden} &
\multicolumn{1}{c|}{Reason} &
 \multicolumn{1}{c|}{(p,q,r,s)} & 
 \multicolumn{1}{c|}{Used} \\ \hline 
\endfirsthead
   1. & & $0 \To p$ & Case 1 & & \\
   2. &  &$C_3(0, 1, p)$,   & Case 1.3.1 & & \\
   & &$C_3(0,p,1)$ & & & \\
   3. &$C_3(0, 3, 2)$,&  & Case 1.3.1 & & \\
   & $C_3(0,2,3)$   & & & & \\
   4. &$1 \From 0$  &   & \ref{lemma:4} &(0,1,2,3) & 1,2 \\
   5. & &$0 \From 1$   & \ref{lemma:1}B &(0,1,2,3) & 1,4 \\
   6. &$0 \From 2$ or $0 \From 3$ & & Primitivity & & 1,5\\
   7. &$0 \From 3$ & & W.l.o.g & & 6\\
   8. & &$3 \From 0$   & \ref{lemma:1}B &(0,3,1,2) & 1,8 \\
   9. & &$3 \From 2$   & \ref{lemma:1}AR &(3,0,2,1) & 3,8 \\
   10. & &$3 \To 2$ or $2 \To 3$   & \ref{lemma:1}BR &(3,2,0,1) & 9 \\
   11. & $2 \To 0$ & & \ref{lemma:5} & (0,2,1,3) & 1 \\
   12. & $3 \To 0$ & & \ref{lemma:5} & (0,3,1,2) & 1 \\
   13. & &$3 \To 0$ or $2 \To 0$   & \ref{lemma:1}A &(3,2,0,1)& 3,10 \\
   & & & &  or (2,3,0,1) & \\
   \hline
      \multicolumn{6}{|l|}{{Now line 13 contradicts lines 11 and 12.}} \\ \hline
\end{longtable}

\begin{remark}
This proof works with some 2-types are forbidden. By assumption, the types $2$ and $3$ are realized. The assumption that the type $1$ is realized only appears in line $4$, which becomes unnecessary if $1$ is forbidden since line 4 is only used for line $5$.
\end{remark}

\begin{longtable}{|c|c|c|c|c|c|} 
\caption{Case 1.3.2} \\
\hline 
\multicolumn{1}{|c|}{Line} &
 \multicolumn{1}{c|}{Realized} & 
 \multicolumn{1}{c|}{Forbidden} &
\multicolumn{1}{c|}{Reason} &
 \multicolumn{1}{c|}{(p,q,r,s)} & 
 \multicolumn{1}{c|}{Used} \\ \hline 
\endfirsthead
   1. & & $0 \To p$ & Case 1 & & \\
   2. &  &$C_3(0, 1, p)$, $C_3(0,2,p)$   & Case 1.3.2 & & \\
   3. &$C_3(0, 3, p)$  &  & Case 1.3.2 & & \\
   4. &$1 \From 0$  &   & \ref{lemma:4} &(0,1,2,3) & 1,2 \\
   5. &$2 \From 0$  &   & \ref{lemma:4} &(0,2,1,3) & 1,2 \\
   6. & &$0 \From 1$   & \ref{lemma:1}B &(0,1,2,3) & 1,4 \\
   7. & &$0 \From 2$   & \ref{lemma:1}B &(0,2,1,3) & 1,5 \\
   8. & &$0 \From 3$   & \ref{lemma:1}AR &(0,1,3,2) & 3,6 \\
   \hline
      \multicolumn{6}{|l|}{{Now $0$ generates an equivalence relation.}} \\ \hline
\end{longtable}

\begin{remark}
By assumption, all 2-types are realized.
\end{remark}

\begin{longtable}{|c|c|c|c|c|c|} 
\caption{Case 1.3.3} \\
\hline 
\multicolumn{1}{|c|}{Line} &
 \multicolumn{1}{c|}{Realized} & 
 \multicolumn{1}{c|}{Forbidden} &
\multicolumn{1}{c|}{Reason} &
 \multicolumn{1}{c|}{(p,q,r,s)} & 
 \multicolumn{1}{c|}{Used} \\ \hline 
\endfirsthead
   1. & & $0 \To p$ & Case 1 & & \\
   2. &$C_3(0, p, 3)$  &  & Case 1.3.3 & & \\
   3. & &$2 \From 1$   & \ref{lemma:1}A &(0,1,2,3) & 1,2 \\
   4. & &$1 \From 2$   & \ref{lemma:1}A &(0,2,1,3) & 1,2 \\
   5. &$1 \From 2$ &   & \ref{lemma:5}R &(2,1,0,3) & 3 \\
   \hline
      \multicolumn{6}{|l|}{{However, line 5 contradicts line 4.}} \\ \hline
\end{longtable}

\begin{remark}
By assumption, all 2-types are realized.
\end{remark}

For Case 2.1, we may assume that $0 \To 1$ and $0 \To 2$ are forbidden, and thus $0 \To 3$ is realized.

\begin{longtable}{|c|c|c|c|c|c|} 
\caption{Case 2.1} \\
\hline 
\multicolumn{1}{|c|}{Line} &
 \multicolumn{1}{c|}{Realized} & 
 \multicolumn{1}{c|}{Forbidden} &
\multicolumn{1}{c|}{Reason} &
 \multicolumn{1}{c|}{(p,q,r,s)} & 
 \multicolumn{1}{c|}{Used} \\ \hline 
\endfirsthead
   1. & &$0 \To 1$, $0 \To 2$ & Case 2.1 & & \\
   2. &$0 \To 3$  & & Case 2 & & \\
   3. &$1 \To 0$  & & \ref{lemma:5} & (0,1,2,3) &1 \\
   4. &$2 \To 0$  & & \ref{lemma:5} & (0,2,1,3) &1 \\
   5. &  &$3 \From 1$, $C_3(0, 1, 3)$  & \ref{lemma:1}A &(0,1,3,2) &1,2 \\
   6. &  &$3 \From 2$, $C_3(0, 2, 3)$  & \ref{lemma:1}A &(0,2,3,1) &1,2 \\
   7. &$3 \From 0$&   & Case 2 & & 5,6 \\
   8. &$1 \From 3$ &   & \ref{lemma:5}R& (3,1,0,2)& 5 \\
   9. &$2 \From 3$ &   & \ref{lemma:5}R& (3,2,0,1)& 6 \\
   10. & &$1 \From 0$ or $0 \From 1$   & \ref{lemma:1}B & (0,1,2,3) & 1 \\
   \hline   
\end{longtable}

We now split into cases based on line 10.
\begin{list}{}{}
\item \textbf{Case 2.1.1:} $1 \From 0$ is forbidden.
\item \textbf{Case 2.1.2:} $0 \From 1$ is forbidden.
\end{list}

\begin{longtable}{|c|c|c|c|c|c|} 
\caption{Case 2.1.1} \\
\hline 
\multicolumn{1}{|c|}{Line} &
 \multicolumn{1}{c|}{Realized} & 
 \multicolumn{1}{c|}{Forbidden} &
\multicolumn{1}{c|}{Reason} &
 \multicolumn{1}{c|}{(p,q,r,s)} & 
 \multicolumn{1}{c|}{Used} \\ \hline 
\endfirsthead
   11. & &$1 \From 0$   & Case 2.1.1 & & 10\\
   12. & &$1 \From 2$, $C_3(1,3,2)$   & \ref{lemma:1}AR &(1,0,2,3) & 4,11 \\
   13. &$1 \To 3$ &   & \ref{lemma:4}R& (3,1,0,2)& 5,12 \\
   14. & &$3 \To 1$   & \ref{lemma:1}BR &(3,1,0,2) & 5,13 \\
   15. & &$3 \To 0$   & \ref{lemma:1}AR &(1,0,3,2) & 8,11 \\
   16. &$3 \To 2$ &  & Case 2 & & 14,15 \\
   17. & &$1 \From 3$ or $3 \To 2$ & \ref{lemma:1}AR & (1,2,3,0) & 12\\
   \hline
      \multicolumn{6}{|l|}{{However, line 17 contradicts line 8 and line 16.}} \\ \hline
\end{longtable}

\begin{longtable}{|c|c|c|c|c|c|} 
\caption{Case 2.1.2} \\
\hline 
\multicolumn{1}{|c|}{Line} &
 \multicolumn{1}{c|}{Realized} & 
 \multicolumn{1}{c|}{Forbidden} &
\multicolumn{1}{c|}{Reason} &
 \multicolumn{1}{c|}{(p,q,r,s)} & 
 \multicolumn{1}{c|}{Used} \\ \hline 
\endfirsthead
   11. & &$0 \From 1$   & Case 2.1.2 & & 10\\
   12. & &$0 \From 2$ or $2 \To 1$, $C_3(0,2,1)$   & \ref{lemma:1}AR &(0,1,2,3) & 11 \\
\hline
\end{longtable}

We now split into cases based on line 12.
\begin{list}{}{}
\item \textbf{Case 2.1.2.1:} $0 \From 2$ is forbidden.
\item \textbf{Case 2.1.2.2:} $2 \To 1, C_3(0,2,1)$ is forbidden.
\end{list}

\begin{longtable}{|c|c|c|c|c|c|} 
\caption{Case 2.1.2.1} \\
\hline 
\multicolumn{1}{|c|}{Line} &
 \multicolumn{1}{c|}{Realized} & 
 \multicolumn{1}{c|}{Forbidden} &
\multicolumn{1}{c|}{Reason} &
 \multicolumn{1}{c|}{(p,q,r,s)} & 
 \multicolumn{1}{c|}{Used} \\ \hline 
\endfirsthead
   13. & &$0 \From 2$  & Case 2.1.2.1 & & 12 \\
   14. &$0 \From 3$ &   & Case 2& & 11,13 \\
   15. & &$3 \To 1$, $C_3(0,3,1)$   & \ref{lemma:1}AR &(0,1,3,2) & 11,14 \\
   16. & &$3 \To 2$, $C_3(0,3,2)$   & \ref{lemma:1}AR &(0,2,3,1) & 13,14 \\
      \hline
         \multicolumn{6}{|l|}{{Now, $0 \cup 3$ generates an equivalence relation.}} \\ \hline
\end{longtable}

\begin{longtable}{|c|c|c|c|c|c|} 
\caption{Case 2.1.2.2} \\
\hline 
\multicolumn{1}{|c|}{Line} &
 \multicolumn{1}{c|}{Realized} & 
 \multicolumn{1}{c|}{Forbidden} &
\multicolumn{1}{c|}{Reason} &
 \multicolumn{1}{c|}{(p,q,r,s)} & 
 \multicolumn{1}{c|}{Used} \\ \hline 
\endfirsthead
   13. & &$2 \To 1$, $C_3(0,2,1)$  & Case 2.1.2.2 & & 12 \\
   14. &$2 \From 0$ &   & \ref{lemma:4}& (0,2,1,3) & 1,6,13 \\
   15. & &$0 \From 2$   & \ref{lemma:1}B &(0,2,1,3) & 1,14 \\
         \hline
            \multicolumn{6}{|l|}{{Now $0 \From 2$ is forbidden, and we may finish as in Case 2.1.2.1.}} \\ \hline
\end{longtable}

\begin{remark}
This proof works with some 2-types forbidden. By assumption, $1$ and $3$ are realized. Assume $2$ is forbidden. Case 2.1.1 ends at line 15 with a contradiction of the Case 2 assumption, since $3 \To 1$, $3 \To 0$, and $3 \To 2$ will all be forbidden. Only lines 4 and 9 depend on 2 being realized, and those are only used in line 12, which would hold anyway if 2 were forbidden. Case 2.1.2.1 works as before, since lines 4 and 9 aren't used anywhere. Also, there is no need for Case 2.1.2.2, since we know $2 \From 0$ is forbidden.
\end{remark}

For Case 2.2, we may assume $0 \To 1$ is forbidden, and thus $0 \To 2$ and $0 \To 3$ are realized.

\begin{longtable}{|c|c|c|c|c|c|} 
\caption{Case 2.2} \\
\hline 
\multicolumn{1}{|c|}{Line} &
 \multicolumn{1}{c|}{Realized} & 
 \multicolumn{1}{c|}{Forbidden} &
\multicolumn{1}{c|}{Reason} &
 \multicolumn{1}{c|}{(p,q,r,s)} & 
 \multicolumn{1}{c|}{Used} \\ \hline 
\endfirsthead
   1. & &$0 \To 1$& Case 2.2 & & \\
   2. &$0 \To 2$, $0 \To 3$  & & Case 2.2 & & \\
   3. &  & $2 \From 1$, $C_3(0,1,2)$ & \ref{lemma:1}A & (0,1,2,3) &1,2 \\
   4. &  & $3 \From 1$, $C_3(0,1,3)$ & \ref{lemma:1}A & (0,1,3,2) &1,2 \\
   5. & $2 \From 0$, $2 \From 3$  & & Case 2.2 & &3\\
   6. & $3 \From 0$, $3 \From 2$  & & Case 2.2 & &4 \\
   7. & & $3 \To 1$, $C_3(2,3,1)$   & \ref{lemma:1}AR & (2,1,3,0) & 3,5 \\
   8. & & $2 \To 1$, $C_3(3,2,1)$   & \ref{lemma:1}AR & (3,1,2,0) & 4,6 \\
   9. & $3 \To 0$, $3 \To 2$  & & Case 2.2 & &7\\
   10. & $2 \To 0$, $2 \To 3$  & & Case 2.2 & &8\\
   11. & &$0 \From 1$, $C_3(3,1,0)$ & \ref{lemma:1}A & (3,1,0,2) &7,9\\
   12. & & $C_3(2,1,0)$ & \ref{lemma:1}A & (2,1,0,3) &8,10\\
   \hline
   \multicolumn{6}{|l|}{{Now $0 \cup 2 \cup 3$ generates an equivalence relation.}} \\ \hline
\end{longtable}

\begin{remark}
By assumption, all 2-types are realized.
\end{remark}

\section{The Imprimitive Case}
 We make an initial case division of the imprimitive case for $\Gamma$ an imprimitive homogeneous 3-dimensional permutation structure. Let $E$ be a minimal non-trivial $\emptyset$-definable equivalence relation.

\begin{list}{}{}
\item \textbf{Case 1:} $E$ is convex with respect to $<_1, <_2, <_3$, and thus is a congruence.
\item \textbf{Case 2:} $E$ is not convex with respect to at least one of $<_1, <_2, <_3$. Without loss of generality, we assume $E$ is not $<_1$-convex.
\end{list}

In Case 1, we may inductively proceed by factoring out $E$, noting that the resulting structure now omits a 2-type, and so $\Gamma$ will be a composition of a homogeneous 3-dimensional permutation structure with one fewer 2-type available and a primitive homogeneous 3-dimensional permutation structure. 

Our goal for Case 2 will be to show that $\Gamma$ is still determined by its restriction to $E$-classes and by the $E$-quotient of the reduct of $\Gamma$ forgetting all orders for which $E$ is non-convex.

The following statement, which is immediate from Theorem \ref{theorem:primclass}, will be important for both cases.

\begin{lemma} \label{lemma:Cgeneric}
Let $E$ be a minimal non-trivial $\emptyset$-definable equivalence relation in a homogeneous 3-dimensional permutation structure, and $C$ be an $E$-class. Then the induced structure on $C$ is generic, modulo the agreement of certain orders up to reversal.
\end{lemma}

We will frequently use the following characterization of genericity.

\begin{proposition}
Let $\Gamma$ be a homogeneous $n$-dimensional permutation structure. Then $\Gamma$ is generic iff for any non-empty open intervals $I_i$ in each order, $<_i$, $\cap_{i=1}^n I_i \neq \emptyset$.
\end{proposition}
\begin{proof}
Genericity of $\Gamma$ is equivalent to the following one-point extension property: given a type $p$ over a finite set $A$ not realized in $A$, $p$ is realized iff its restriction to each individual order is realized by an element not in $A$. The restriction of $p$ to an order $<_i$ specifies a $<_i$-interval with endpoints in $A \cup \set{\pm \infty}$, which is open since $p$ is not realized in $A$. This interval is non-empty exactly when the restriction has a realization not in $A$.
\end{proof}

\subsection{Convex Closure}
In this section, we show $E$-classes are $<_1$-dense in their $<_1$-convex closures, and the $<_1$-convex closure of $E$ is an equivalence relation. The arguments we present depend heavily on the type structure in the case $k=3$, although in a few cases a step where our argument depends on $k=3$ could have been carried out in greater generality.

\begin{lemma} \label{lemma:norev}
Let $E$ be a minimal non-trivial $\emptyset$-definable equivalence relation in a homogeneous finite-dimensional permutation structure, and $C, C'$ be distinct $E$-classes. Then no 2-type $p$ is realized in both $C \times C'$ and $C' \times C$.
\end{lemma}
\begin{proof}
Let $a, b \in C$, $a', b' \in C'$, such that $a \ra{p} b'$ and $a' \ra{p} b$. Let $b \ra{q} b'$, and note that $p \neq q$, since otherwise transitivity would force $a' \ra{p} b'$ and so $p \subset E$. By homogeneity, there is an automorphism sending $(a, b')$ to $(a', b)$. Thus there must be some $c \in C$ such that $b' \ra q c$. But then by transitivity $b \ra q c$, which is a contradiction.
\end{proof}

\begin{definition}
Let $\widetilde E$ be the $<_1$-convex closure of $E$, i.e. $a \widetilde E b$ if there exists a $c$ such that $aEc$ and $b$ is $<_1$-between $a$ and $c$. Given an $E$-class $C$, $\widetilde C$ is the $<_1$-convex closure of $C$.
\end{definition}

\begin{notation*}
For the rest of \S 4, we fix notation, by reversing and switching orders as needed, so that the 2-type $\ra{123}$ is contained in $E$, and if $E$ contains another 2-type besides $\ra{123}$ and its opposite then it contains $\ra{23}$.
\end{notation*}

\begin{lemma}\label{lemma:3cross}
Let $E$ be a minimal non-trivial $\emptyset$-definable equivalence relation in a homogeneous 3-dimensional permutation structure, and $C$ be an $E$-class. Let $a_1, a_2 \in C$, $b \not\in C$, and $a_1 <_1 b <_1 a_2$. Then $tp(a_1, b) = \ra{12}$, $ tp(b, a_2) = \ra{13}$, or $tp(a_1, b) = \ra{13}, tp(b, a_2) = \ra{12}$.
\end{lemma}
\begin{proof}
If $E$ contains $\ra{123}$ and $\ra{23}$, the conclusion follows by Lemma \ref{lemma:norev} and the fact that only 4 2-types remain.

Otherwise, we have that $a_1 \ra{123} a_2$. Since we cannot have $a_1 \ra{123} b$, there is some $i$ such that $b <_i a_1 <_i a_2$, and so $b \ra{1i} a_2$. Thus, there is a unique $j$ such that $a_2 <_j b$, so $a_1 <_j b$. Thus $a_1 \ra{1j} b$ and $b \ra{1i} a_2$.
\end{proof}

\begin{corollary}\label{cor:contain}
Let $E$ be a minimal non-trivial $\emptyset$-definable equivalence relation in a homogeneous 3-dimensional permutation structure, and $C$ be an $E$-class. Suppose $b \in \widetilde C \bs C$. Then $b/E \subset \widetilde C$.
\end{corollary}
\begin{proof}
Take $a_1, a_2 \in C$ such that $a_1 <_1 b <_1 a_2$. By Lemma \ref{lemma:3cross}, we may suppose without loss of generality that $tp(a_1, b) = \ra{12}, tp(b, a_2) = \ra{13}$.

Take $b' \in b/E$. If $b' >_1 a_2$, then $b <_1 a_2 <_1 b'$ and $tp(b, a_2) = \ra{13}$, so by Lemma \ref{lemma:3cross}, $tp(a_2, b') = \ra{12}$. By homogeneity, there is an automorphism $\phi$ sending $(a_1, b)$ to $(a_2, b')$, so $b'$ is $<_1$-between $a_2$ and $\phi(a_2)$. The case where $b' <_1 a_1$ is nearly identical.
\end{proof}

\begin{corollary}\label{cor:samecross}
Let $E$ be a minimal non-trivial $\emptyset$-definable equivalence relation in a homogeneous 3-dimensional permutation structure, and $C$ be an $E$-class. Let $a, a' \in C$, with $a' \ra{123} a$. For any $b \not\in C$, if $b <_1 a'$ or $a <_1 b$, then $tp(a, b) = tp(a', b)$
\end{corollary}
\begin{proof}
We only treat the case $a <_1 b$, since the other case is similar.

Suppose $tp(a, b) = \ra{1x}$. By transitivity, $a' <_1 b$, $a' <_x b$. Since we cannot have $tp(a', b) = \ra{123}$, we are done. 

Now suppose $tp(b, a) = \ra{23}$. By transitivity, $a' <_1 b$. However, we cannot have $tp(a', b) = \ra{12}$ or $\ra{13}$, since by Lemma \ref{lemma:3cross} there would be some $a'' \in C$ such that $a'' >_1 b$, and then applying Lemma \ref{lemma:3cross} again, we would have that $tp(a, b)$ would also be $\ra{12}$ or $\ra{13}$.
\end{proof}

\begin{corollary} \label{cor:min}
Let $E$ be a minimal non-trivial $\emptyset$-definable equivalence relation in a homogeneous 3-dimensional permutation structure. Then any non-trivial $\emptyset$-definable equivalence relation contains $E$.
\end{corollary}
\begin{proof}
Consider the equivalence relation generated by a 2-type $p$, and without loss of generality assume $<_1$ holds in $p$. Given $a, b$ such that $a \ra{p} b$, find $b'$ such that $b \ra{123} b'$. By Corollary \ref{cor:samecross}, $a \ra{p} b'$, so $p$ generates $\ra{123}$. 

If $\ra{23} \subset E$, so $p = \ra{1x}$, then run the above argument with $b \ra{23} b'$. By transitivity, $a <_x b'$, so $a \ra{1x} b'$, and $\ra{23}$ is generated by $p$ as well.
\end{proof}

We note that much of the proof of the following lemma is concerned with ruling out a plausible configuration in which given $E$-classes $C, C_1$ such that $C_1 \subset \widetilde C$, then $C_1$ defines a non-trivial $<_1$-Dedekind cut in $C$. Although the type structure is too constrained to allow this with 3 orders, it seems possible that it may occur with more orders.

\begin{lemma} \label{lemma:dense}
Let $E$ be a minimal non-trivial $\emptyset$-definable equivalence relation in a homogeneous 3-dimensional permutation structure, and $C$ be an $E$-class. Then $C$ is $<_1$-dense in $\widetilde C$.
\end{lemma}
\begin{proof}
Given an $E$-class $C$ and an element $a$, we use $\widehat a(C)$ to denote the $<_1$-Dedekind cut defined by $a$ in $C$.

Let $a, b \in \widetilde C$, with $a <_1 b$, and suppose $\widehat a(C) = \widehat b(C)$. We first show we may suppose that $a/E = b/E$ and $a \ra{123} b$.

  Trivially, we cannot have $a, b \in C$. Now assume only one of $a, b \in C$, say $a$. Then $a$ is a maximal element of the cut $\widehat b(C)$. But given any $d \in \widetilde C\bs C$, $\widehat d(C)$ has no maximal or minimal elements; otherwise, the elements of $C$ would realize at least 3 types over $d$, but there are only 2 realized types by Lemma \ref{lemma:3cross}. Thus $a, b \in \widetilde C \bs C$.

Now suppose $a/E = b/E$, but $a \la{23} b$. By the genericity of $C$, there is a $b'$ in the $<_1$-interval $(a, b)_{<_1}$ such that $a \ra{123} b'$, so we may replace $b$ by $b'$.

\begin{claim}
Suppose $C_1 = a/E \neq b/E = C_2$. Then there exists $a' \in C_1$ such that $a' \ra{123} a$ and $\widehat a(C) = \widehat {a'}(C)$
\end{claim}
\begin{claimproof}
Let $a' \ra{123} a$. Since $a <_1 b$, by Corollary \ref{cor:samecross} $tp(a, b) = tp(a', b)$. Since by Lemma \ref{cor:contain}, $a/E \in \widetilde C$, there is a $c \in C$ such that $c <_1 a'$, so by Corollary \ref{cor:samecross} $tp(a, c) = tp(a', c)$. Thus $(a, b, c) \cong (a', b, c)$, so by homogeneity there is an automorphism fixing $c$ and taking $(a, b)$ to $(a', b)$. Thus $\widehat {a'}(C) = \widehat b(C) = \widehat a(C)$.
\end{claimproof}

In this case, we may then replace $a, b$ by $a', a$. 

Thus, we may now suppose that $a/E = C_1 = b/E$ and $a \ra{123} b$.

\begin{claim}
$\widehat c(C)$ is independent of the choice of $c \in C_1$.
\end{claim}
\begin{claimproof}
Consider $x, y \in C_1$, and, using the genericity of $C_1$, find $c_1, c_2 \in C_1$ such that $c_1 <_1 x,y <_1 c_2$ and $c_1 \ra{123} c_2$.

Take $z \in C$, with $z <_1 a,c_1$. By Corollary \ref{cor:samecross}, $(z, a, b) \cong (z, c_1, c_2)$. Thus, since $\widehat a(C) = \widehat b(C)$, we have $\widehat c_1(C) = \widehat c_2(C)$, so $\widehat x(C) = \widehat y(C)$.
\end{claimproof}

Without loss of generality, we now assume $C <_2 C_1$, so by Lemma \ref{lemma:3cross}, the types realized in $C \times C_1$ are $\ra{12}$ and $\la{13}$. Thus by homogeneity, given any $E$-classes $C, C'$, if $\ra{12}$ or $\la{13}$ is realized in $C \times C'$, then $C'$ defines a $<_1$-Dedekind cut in $C$; if neither these types nor their opposites are realized, then the only remaining types are $\ra{23}$ and $\la{23}$, and by Lemma \ref{lemma:norev} exactly one of them is realized, so neither class is in the $<_1$-convex closure of the other. In particular, $E$-classes are $<_2, <_3$-convex.

Note that if every $E$-class $C' \subset \widetilde C$ such that $C <_2 C'$ defined the same $<_1$-Dedekind cut in $C$, then $C$ would have an $\emptyset$-definable partition, contradicting the minimality of $E$.

\begin{claim}
Both factors of the $(12,23,13)$-majority diagram, displayed below (with the edge $x \ra{23} z$ not drawn), are realized in $\Gamma$.

\begin{figure}[h]
\begin{diagram}
& & \overset{x}{\bullet} & & \\
& \ruTo^{12} & \dTo_{23}  & \rdTo^{23} & \\
a \odot &\rTo^{12} & \bullet_{y}  & \rTo^{13} & \odot b \\
& \rdTo_{23} & \dTo_{23} & \ruTo_{13} & \\
& & \underset{z}{\bullet} & & \\
\end{diagram}
\caption{}
\label{fig:15}
\end{figure}
\end{claim}
\begin{claimproof}
We only prove the first factor is realized, since the argument for the second is nearly identical. First, as shown in Figure \ref{fig:16}, the first factor is the unique amalgam of the following 3-types, so it suffices to show these are realized.

\begin{figure}[h]
\begin{diagram}
& & \overset{x}{\odot} & & \\
& \ruTo^{12} & \dTo_{23}   \\
a \bullet &\rTo^{12} & \bullet_{y}  \\
& \rdTo_{23} & \dTo_{23}\\
& & \underset{z}{\odot} & & \\
\end{diagram}
\caption{}
\label{fig:16}
\end{figure}

For the 3-type $(a,x,y)$ from the diagram, let $a/E = C$.  Take distinct $E$-classes $C', C'' \subset \widetilde C$ such that $C <_2 C'<_2 C''$ and $C'$ and $C''$ define distinct $<_1$-cuts in $C$. Then there are $x \in C', y \in C''$ realizing the triangle $(a,x,y)$.

For the 3-type $(a,y,z)$, we claim it is the unique amalgam of the diagram in Figure \ref{fig:17}.

\begin{figure}[h]
\begin{diagram}
 y \odot & &  & & \odot z  \\
& \luTo_{12} &  & \ruTo_{23} & \\
& & {\bullet}_a & & \\
\end{diagram}
\caption{}
\label{fig:17}
\end{figure}

By transitivity, $y <_3 z$ and $z <_1 y$, so the possible completions are $y \ra{23} z$ and $z \ra{12} y$. However, if $z \ra{12} y$, then $y/E$ defines a $<_1$-Dedekind cut in both $a/E$ and $z/E$, but $z/E <_1 a/E$, which is a contradiction. Thus the only allowed completion is $y \ra{23} z$.
\end{claimproof}

We are forced to complete the $(12,23,13)$-majority diagram by $a \ra{123} b$, so that $aEb$. However, $a \ra{12} x \ra{23} b$ violates the requirement that $E$-classes are $<_2$-convex. Thus $\Gamma$ is not homogeneous.
\end{proof}

\begin{proposition}
Let $E$ be a minimal non-trivial $\emptyset$-definable equivalence relation in a homogeneous 3-dimensional permutation structure, and $C$ be an $E$-class. Then $\widetilde E$ is an equivalence relation.
\end{proposition}
\begin{proof}

Let $C$ be an $E$-class. By Corollary \ref{cor:contain}, $\widetilde C$ is a union of $E$-classes. Now suppose $C' \subset \widetilde C$ is an $E$-class. By Lemma \ref{lemma:dense} there are $c_1, c_2 \in C'$ such that $\widehat c_1(C) \neq \widehat c_2(C)$, and applying Corollary \ref{cor:contain} again we see $C \subset \widetilde {C'}$. Thus $\widetilde E$ defines a partition.
\end{proof}

\begin{corollary} \label{cor:congconv}
$\widetilde E$ is a congruence, $E$-classes are $(<_2,<_3)$-convex, ${<_2} = {<_3}$ on $E$-classes, and ${<_2} = {<_3^{\opp}}$ between $E$-classes in the same $\widetilde E$-class.
\end{corollary}

\subsection{Reduction via Quotients}
Since $\widetilde E$ is a congruence by Corollary \ref{cor:congconv}, it suffices to consider the case $\widetilde E = \bbone$, since we may otherwise consider the restriction ${\Gamma \upharpoonright \widetilde E}$. For this subsection, we work with $k$-dimensional permutation structures.

We now aim for the following lemmas. The first implies that $\Gamma$ is determined by its restriction to $E$-classes and the reduct of $\Gamma/E$ forgetting all orders that are not $E$-convex. The second allows us to carry out our induction by showing that the above reduct of $\Gamma/E$ must be homogeneous.

The following lemma is more naturally stated in the language of subquotient orders, but as it is the concluding step in the classification of certain permutation structures, we give it in a form appropriate for its intended application.

\begin{lemma} \label{lemma:uniqueextension}
Let $(\Gamma^*, <_1^*, ..., <_\ell^*)$ be homogeneous.Let $k \geq \ell$, and partition $[k]$ as $\cup_{i \leq m} I_i$, such that each $I_i$ contains at most one $j \geq \ell+1$.
 Then there exists a homogeneous structure $(\Gamma, E, <_1, ..., <_k)$, unique up to isomorphism, with the following properties.
\begin{enumerate}
\item $E$-classes are $<_1, ..., <_\ell$-convex and $<_{\ell+1}, ..., <_k$-dense.
\item $(\Gamma/E, <_1, ..., <_\ell) \cong (\Gamma^*, <_1^*, ..., <_\ell^*)$
\item ${<_j \upharpoonright_E} = {<_{j'}\upharpoonright_E}$ for $j, j'$ in a given $I_i$, and the induced structure on any $E$-class $C$ is fully generic, modulo the identification of orders in the same $I_i$.
\end{enumerate}
\end{lemma}


\begin{lemma} \label{lemma:homquotient}
Let $\Gamma$ be a homogeneous $k$-dimensional permutation structure.  Let $E$ be a minimal non-trivial $\emptyset$-definable equivalence relation in $\Gamma$, and suppose $E$-classes are $<_i$-convex for $1 \leq i \leq \ell$ and $<_i$-dense for $\ell+1 \leq i \leq k$. Suppose each $E$-class is generic, modulo the agreement of certain orders up to reversal. Then $(\Gamma/E, <_{1}, ..., <_\ell)$ is homogeneous.
\end{lemma}

The following lemma is not necessary for the case $k=3$, since $E$ is only dense with respect to one order.

\begin{lemma}\label{lemma:denserestrict}
Suppose $(\Gamma, <_1, ..., <_k)$ is homogeneous. Let $E$ be a minimal non-trivial $\emptyset$-definable equivalence relation in $\Gamma$, and $C$ be an $E$-class. Suppose $C$ is generic, modulo the agreement of certain orders up to reversal. Further suppose that $C$ is $<_i$-convex for $1 \leq i \leq \ell$ and $<_i$-dense for $\ell+1 \leq i \leq k$. Then
\begin{enumerate}
\item If $C_1, C_2$ are $E$-classes, then $C_1$ remains homogeneous after naming $C_2$.
\item If $i, j \geq \ell+1$, and ${<_i \upharpoonright_E} = {<_j \upharpoonright_E}$, then ${<_i} = {<_j}$.
\end{enumerate}
\end{lemma}
\begin{proof} \hspace{.1 cm} \newline
$(1)$ Given a finite $A \subset C_1$ and $i \geq \ell+1$, let $B_i = \set{x \in C_2 | A <_i x}$. Each such $B_i$ is a $<_i$-terminal segment of $C_2$, so by genericity their intersection is non-empty.

Now, consider $A_1 \cong A_2$ finite substructures of $C_1$. Let $A = A_1 \cup A_2$, and choose a $b$ in the intersection of the corresponding $B_i$. By homogeneity, there is an automorphism taking $A_1b$ to $A_2b$ and fixing $b$, hence $C_2$.

$(2)$ Suppose this is false, as witnessed by $<_i, <_j$. We consider $E$-classes as ordered sets with respect to the common restriction of these orders.

Take $a, b$ with $a <_i b$ and $b <_j a$, and  let $C_1 = a/E$ and $C_2 = b/E$. Let
$$I_a = \set{x \in C_2| a<_i x, x<_j a}, J_a = \set{x \in C_1| I_a \cap I_x \neq \emptyset}$$
Note that these are intervals in $C_2$ and $C_1$, respectively.

\begin{claim*}
$J_a = \set{a}$
\end{claim*}
\begin{claimproof}
By density and genericity, there are $b_1, b_2 \in C_2$ such that $b_1 <_{i,j} a <_{i, j} b_2$, so $I_a \subset (b_1, b_2)$. Then find $a_1, a_2 \in C_1$ such that $a_1 <_{i, j} b_1, b_2 <_{i, j} a_2$, so $J_a \subset (a_1, a_2)$.

Thus $J_a$ is $(a, C_2)$-definable and $<_{i, j}$-bounded. By $(1)$, $J_a$ is $a$-definable in $C_1$ and $<_{i, j}$-bounded, so $J_a = \set{a}$ by genericity of $C_1$.
\end{claimproof}

If there were some $b' \in I_a$ with $b' \neq b$, then by density, we could find some $a' \in C_1$ $<_i$-between $b$ and $b'$, and so would have $a' \in J_a$. Thus $I_a = \set{b}$. But by density there is a $b' \in C_2$ $<_i$-between $a$ and $b$, so $b' \in I_a$, which is a contradiction.
\end{proof}

Given $(\Gamma, <_1, ..., <_k)$ homogeneous such that no orders agree up to reversal, with $E$-classes $<_i$-convex for $1 \leq i \leq \ell$ and $<_i$-dense for $\ell+1 \leq i \leq k$, we will prefer to work in the quantifier-free interdefinable reduct $\Gamma^{\red} = (\Gamma, {<_{i_1}',} ..., {<_{i_m}',} {<_1'',} ...,{<_\ell'',} {<_{\ell+1},} ..., <_k)$ obtained as follows.
\begin{enumerate}
\item For $1 \leq i \leq \ell$, decompose $<_i$ into two subquotient orders: $<_i'$ from $\bbzero$ to $E$ and $<_i''$ from $E$ to $\bbone$. 
\item For each $i \geq \ell+1$, add the restriction $<_i \upharpoonright_E$ to the language as $<_i'$.
\item Consider the set of all $<_i'$. Many of these subquotient orders may be equal up to reversal, so pick one representative from each class and forget the rest. By Lemma \ref{lemma:denserestrict}, each class can contain at most one $<_i'$ with $i \geq \ell+1$, in which case this is taken as the representative.
\item Forget the $<_i'$ for $i \geq \ell+1$. 
\end{enumerate}

We now prove a 1-point extension property, which shows that to realize a type $p$ in an $E$-class $C$ , it is sufficient that the restriction of the type to each subquotient order is individually realized. 

\begin{lemma}\label{lemma:1point}
Let $(\Gamma, <_1, ..., <_k)$ be homogeneous such that no orders agree up to reversal. Let $E$ be a minimal non-trivial $\emptyset$-definable equivalence relation in $\Gamma$, and $C$ be an $E$-class. Suppose the induced structure on $C$ is generic, modulo the agreement of certain orders up to reversal. Suppose that $C$ is $<_i$-convex for $1 \leq i \leq \ell$ and $<_i$-dense for $\ell+1 \leq i \leq k$. We now work in $\Gamma^{\red}$.

Let $A \subset \Gamma^{\red}$ be finite, and $p$ a 1-type over $A$ not realized in $A$. Then $p$ is realized in a given $E$-class $C$ by a point not in $A$ iff the following hold.
\begin{enumerate}
\item $p \upharpoonright (<_1'', ..., <_\ell'')$ is realized by $C$ in $\Gamma/E$.
\item For each $<_i'$,  $p \upharpoonright <_i'$ is realized in $C \bs A$.
\item For $j \geq \ell+1$, $p \upharpoonright <_j$ is realized by some element not in $A$.
\item $p$ does not contain the formula ``$x = a$'' for any $a \in A$.
\end{enumerate}
\end{lemma}
\begin{proof}
These conditions are clearly necessary. We will prove they suffice. By condition $(1)$, all of $C$ satisfies $p\upharpoonright (<_1'', ..., <_\ell'')$. List all the subquotient orders from $\bbzero$ to $E$ together with $<_i$ for $i \geq \ell+1$ as $<_1^*, ..., <_n^*$, and let $p_i = p \upharpoonright <_i^*$. It now suffices to show $p_i$ contains a non-empty open $<_i^*$-interval of $C$, since then by the genericity of $C$ there will be some point in their intersection, which thus realizes $p$. 

In the case $<_i^*$ is a subquotient order from $\bbzero$ to $E$, by condition $(2)$ some point in $C$ realizes $p_i$ restricted to parameters outside of $C$, and so all of $C$ does; again by condition $(2)$, $p_i$ restricted to parameters inside of $C$ then contains an open interval of $C$. In the case ${<_i^*} = {<_j}$ for $j \geq \ell+1$, condition $(3)$ implies $p_i$ contains a non-empty open interval in $\Gamma$; since $E$-classes are $<_i^*$-dense, this interval meets $C$ in a non-empty open interval.
\end{proof}

\begin{proof}[Proof of Lemma \ref{lemma:homquotient}]
Let $\bar A \cong \bar B$ be finite subsets of $(\Gamma/E, <_1, ..., <_\ell)$. We lift $\bar A$ to $A \subset (\Gamma, <_1, ..., <_k)$, and look for an automorphism moving $A$ to a set covering $B$.

We proceed by induction on $|\bar A|$, and so consider $A = A_0 \cup \set{a}$ with $\bar a \not\in \bar A_0$, $\bar B = \bar A_0 \cup \set{C}$ for some $E$-class $C \not\in \bar A_0$.

Let $p = tp(a/A_0)$. We will now work in $\Gamma^{\red}$ and use Lemma \ref{lemma:1point} to find a realization of $p$ in $C$. Condition $(1)$ is equivalent to $\bar A \cong \bar B$. Since $\bar a \not\in \bar A_0$, $A_0 \cap C = \emptyset$, so $(p \upharpoonright <_i') \upharpoonright A$ simply says $x$ is not $<_i'$-related to any $a \in A$, which will be true for every $x \in C$. Finally, since $\bar a \not\in \bar A_0$, $a \not\in A_0$, so $a$ witnesses condition $(3)$.
\end{proof}

\begin{proof}[Proof of Lemma \ref{lemma:uniqueextension}]
For existence, let $\Gamma$ be the composition $\Gamma^*[C]$, where $C$ only carries the equality relation, and let $E$ be the corresponding equivalence relation. Note that each $<_i^*$ is now a subquotient order from $E$ to $\bbone$. For $1 \leq i \leq m-(k-\ell)$, add a generic subquotient order $<'_i$ from $\bbzero$ to $E$. For $\ell+1 \leq i \leq k$, add a generic linear order $<_i$. We may then define the specified convex orders $<_i$ for $1 \leq i \leq \ell$ as compositions of the $<_i^*$ with the $<_j'$ or the restrictions to $E$ of the $<_n$ for $\ell+1 \leq n \leq k$. 

For uniqueness, suppose we have a structure $(\Gamma', <_1, ..., <_k)$ satisfying the conditions. We will show $(\Gamma')^{\red}$ has the same finite substructures as the $\Gamma^{\red}$ we constructed above; as they are both homogeneous, they will thus be isomorphic.

As all the subquotient orders added to construct $\Gamma^{\red}$ were added generically, every finite substructure of $(\Gamma')^{\red}$ is a substructure of $\Gamma^{\red}$. We proceed by induction on the size of the substructure, so let $A \cup \set{a}$ be a finite substructure of $\Gamma^{\red}$, such that $A$ is a substructure of $(\Gamma')^{\red}$. We will use Lemma \ref{lemma:1point} to show $p = tp(a/A)$ is realized in $(\Gamma')^{\red}$.

We may assume $a \not\in A$, otherwise we are done, so condition $(4)$ is satisfied. As (suitable reducts of) $\Gamma^{\red}/E$ and $(\Gamma')^{\red}/E$ both are isomorphic to $\Gamma^*$, and as $a/E$ realizes $p \upharpoonright (<_1'', ..., <_\ell'')$ in the former, there is some $E$-class $C$ realizing it in the latter, so condition $(1)$ is satisfied. For condition $(2)$,  again since the quotient structures are isomorphic, we may pick $C$ such that for each $b \in A$, $C = b/E$ iff $a/E = b/E$. Thus, we are only concerned about $(p \upharpoonright <_i') \upharpoonright (A \cap C)$; but as this restricted type doesn't violate transitivity, it is realized in $C$ since $<_i'$ is dense on $C$. Finally for condition $(3)$, we again have that $p \upharpoonright <_j$ doesn't violate transitivity, and so is realized by some element not in $A$ since $<_j$ is dense on $(\Gamma')^{\red}$.
\end{proof}



\begin{remark}
Lemma \ref{lemma:uniqueextension} is also true if $(3)$ is relaxed to allow certain restrictions to be the reversals of others. The only case that isn't immediate is if we require $<_i \upharpoonright_E = (<_j \upharpoonright_E)^{\opp}$ for $<_i, <_j$ dense. But then ${<_i} = {<_j^{\opp}}$ by Lemma \ref{lemma:denserestrict}.
\end{remark}

\subsection{The Imprimitive Catalog}
We now classify the imprimitive homogeneous structures. Listing all these structures in the language of linear orders yields a mob of examples, since (pieces of) orders may be reversed, which orders a given order agrees with may differ for the various pieces of that order, and orders may be permuted. Thus, we present the structures up to definable equivalence, and do so in a language of subquotient orders, each of which is generic, and equivalence relations.

We first classify the imprimitive homogeneous 3-dimensional permutation structures $(\Gamma, E, <_1, <_2, <_3)$ in which $\widetilde E = \bbone$, so $\Gamma$ has no non-trivial $\emptyset$-definable congruence. By Corollary \ref{cor:congconv} and Lemmas \ref{lemma:homquotient} and \ref{lemma:uniqueextension}, $\Gamma$ is determined by $(\Gamma/E, <_2)$ and $(\Gamma \upharpoonright_E, <_1, <_2)$, which are themselves primitive homogeneous. There are thus two possibilities.

\begin{enumerate}
\item ($<_1 \upharpoonright_E \neq <_2 \upharpoonright_E$) $\Gamma$ may be presented as $(\Gamma, E, (<_i')_{i=1}^3)$ with $<_1'$ from $\bbzero$ to $\bbone$, $<_2'$ from $\bbzero$ to $E$, and $<_3'$ from $E$ to $\bbone$.
\item (${<_1 \upharpoonright_E} = {<_2 \upharpoonright_E}$) $\Gamma$ may be presented as $(\Gamma, E, (<_i')_{i=1}^2)$ with $<_1'$ from $\bbzero$ to $\bbone$ and $<_2'$ from $E$ to $\bbone$.
\end{enumerate}
$(1)$ is just a generic order added to Cameron's imprimitive homogeneous permutation, i.e. $\Q^2$ with an equivalence relation for agreement in the first coordinate and a lexicographic order, while $(2)$ is the structure described in Example \ref{ex:complex3} in Section 2. 

Also note that when presented in the language of 3 linear orders, $(1)$ uses all 8 2-types, while $(2)$ only uses 6 of them. Thus $(1)$ cannot appear as a factor in a composition, while $(2)$ can.

If $\Gamma$ has a non-trivial $\emptyset$-definable congruence, then it is a composition, whose factors are either primitive or one of the above structures. Below, let $\Gamma^{(g)}_i$ to denote the generic $i$-dimensional permutation structure

If all of the factors are primitive, then each factor is interdefinable with $\Gamma^{(g)}_i$ for $i \in \set{1, 2}$. Each such factor contributes $2^i$ 2-types. As there are at most 8 2-types available, we get at most the following structures.

\begin{enumerate}
\item[(3)] For any multisubset $I \subset \set{1, 2}$ such that $|I|>1$ and $\sum_{i \in I} 2^i \leq 8$, $\Gamma$ is the composition in any order of $\Gamma^{(g)}_i$ for $i \in I$.
\end{enumerate}

Finally, if one of the factors is imprimitive, we noted earlier it must be $(2)$. There are only 2 2-types remaining, so the other factor must be $\Gamma^{(g)}_1$.

\begin{enumerate}
\item[(4)] Let $\Gamma^*$ be the structure from $(2)$. Then $\Gamma = \Gamma^*[\Gamma^{(g)}_1]$ or $\Gamma^{(g)}_1[\Gamma^*]$.
\end{enumerate}

For all of these structures we have only shown that at most 8 2-types are realized, but it is easy to check that each structure can be presented in a language of 3 linear orders by taking restrictions and compositions of the subquotient orders, which concludes our derivation of the catalog.

 This last step prompts the following special case of Question \ref{qu:represent}.

\begin{question}
Let $\Gamma$ be a finite-dimensional permutation structure, with a linear lattice of $\emptyset$-definable equivalence relations. If $\Gamma$ has at most $2^k$ non-trivial 2-types, can $\Gamma$ be presented as a $k$-dimensional permutation structure?
\end{question}

We remark that the linearity hypothesis is necessary, since the full product $\Q^2$ (see Example \ref{ex:fullQ2}) only has 8 non-trivial 2-types, but requires 4 linear orders.

\appendix
\section{Proof of Theorem \ref{theorem:amalg}}
We first repeat the theorem we wish to prove.

\newtheorem*{theorem:amalg}{\bf{Theorem \ref{theorem:amalg}}}
\begin{theorem:amalg}
Let $\Lambda$ be a finite distributive lattice, and $\Gamma$ the generic $\Lambda$-ultrametric space. For each meet-irreducible $E \in \Lambda$, fix a function $f_E: \set{F \in \Lambda| E<F} \to \N$. Then there is a homogeneous expansion of $\Gamma$, which is generic in a natural sense, adding, for each meet-irreducible $E \in \Lambda$ and $F>E$, $f_E(F)$ subquotient orders from $E$ to $F$.

To be more precise, the following holds. Let $\AA^*$ be the class of finite structures $(A,d,\set{<_{E_i, j}}_{j=1}^{n_i})$ satisfying the following conditions.
\begin{itemize}
\item $(A,d)$ is  a $\Lambda$-ultrametric space.
\item $<_{E_i, j}$ is a subquotient order with bottom relation $E_i$, for some meet-irreducible $E_i \in \Lambda$, and top relation $F_{i, j} \in \Lambda$.
\end{itemize}
Then $\AA^*$ is an amalgamation class, and its \fraisse limit is the desired expansion if $\set{E_i}$ and $\set{F_{i,j}}$ are chosen to match $f_E$ from above.
\end{theorem:amalg}

Although we could prove this theorem by straightforward modifications of Lemmas 3.7 and 3.8 of \cite{Braun}, we choose to present a different take on the proof here.

\begin{definition} [\cite{Braun}, {Definition 2.3}]
Consider an amalgamation diagram of $\Lambda$-ultrametric spaces with base $B$. Let $x$ and $y$ be extension points in different factors, and for each $b_i \in B$ let $d(x, b_i) = e_i$ and $d(y, b_i) = e'_i$. \textit{Pre-canonical amalgamation} is the amalgamation strategy assigning $d(x, y) = \bigwedge_i(e_i \vee e'_i)$. \textit{Canonical amalgamation} is the strategy of pre-canonical amalgamation, followed by identifying $x$ and $y$ if $d(x, y) = \bbzero$. 
\end{definition}

\begin{proposition}[\cite{Braun}, {Proposition 2.4}]
 Let $\Lambda$ be a distributive lattice, and let $\KK$ be the class of all finite $\Lambda$-ultrametric spaces. Then $\KK$ is an amalgamation class, and any amalgamation diagram can be completed by canonical amalgamation.
 \end{proposition}

\begin{definition}
Let $X$ be a set equipped with a binary relation $R$ and an equivalence relation $E$. We say that $E$ is a \textit{$R$-congruence} if $E(x, x')$ and $E(y, y')$ implies that $R(x, y)$ iff $R(x', y')$.
\end{definition}

\begin{proof}[Proof of Theorem \ref{theorem:amalg}]
Like linear orders, subquotient orders may be amalgamated independently, so we may assume $n=1$, and we will call the only subquotient order $<_E$, and call its bottom relation $E$ and top relation $F$.

We first introduce some notation. We define the relations $\preceq_E$ and $\arr$ on $\AA^*$-structures by 
\begin{enumerate}
\item $a \preceq_E b \Leftrightarrow (d(a, b) \leq E) \vee (a <_E b)$
\item $a \arr b \Leftrightarrow \exists x(a \preceq_E x \preceq_E b) \wedge (d(a, b) \not\leq E)$.
\end{enumerate}

We will make use of the following properties of $\preceq_E$ on $A^*$-structures.
\begin{enumerate}
\item If $a \preceq_E b <_E c$ or $a <_E b \preceq_E c$, then $a <_E c$.
\item $\preceq_E$ is transitive.
\item If $a \preceq_E b \preceq_E c$ and $d(a, c) \leq E$, then $d(a, b), d(a, c) \leq E$.
\item If $a \preceq_E b \preceq_E a$, then $d(a, b) \leq E$.   
\end{enumerate}

Property $(1)$ follows from the fact that $E$ is a $\preceq_E$-congruence. Properties $(2)$ and $(3)$ follow from $(1)$, and $(4)$ is a special case of $(3)$.

It suffices to show that $\AA^*$ contains solutions to all two-point amalgamation problems
$A_0^*\includedin A_1^*,A_2^*$, $A_i^*=A_0^*\union \{a_i\}$ for $i=1,2$.

Let $A$ be the extension of the free amalgam given by determining $d(a_1, a_2)$ by pre-canonical amalgamation. Either $<_E$ is already a subquotient order with bottom relation $E$ and top relation $F$, or we need to extend it to one by determining either $a_1 <_E a_2$ or $a_2 <_E a_1$. We break this into three cases.

\begin{claim}
Suppose $d(a_1, a_2) \leq E$. Then for $x \in A^*_0$, we have
$$a_1 <_E x \Longleftrightarrow a_2 <_E x$$
In particular, $<_E$ is a subquotient order on $A$ from $E$ to $F$.
\end{claim}
\begin{claimproof}
As $E$ is meet-irreducible, if pre-canonical amalgamation gives $d(a_1, a_2) \leq E$, then there is a $y \in A_0^*$ such that $d(a_1, y), d(a_2, y) \leq E$.

By the fact that $E$ is a $<_E$-congruence, we get $a_1 <_E x \Longleftrightarrow y <_E x \Longleftrightarrow a_2 <_E x$. This proves the first part of the claim, and the second part follows immediately.
\end{claimproof}

We also note that if $d(a_1, a_2) = \bbzero$, then by the above claim $A_1 \cong A_2$, so we may amalgamate by identifying $a_1$ with $a_2$.

\begin{claim}
Suppose $d(a_1, a_2) \not\leq F$. Then $<_E$ is a subquotient order on $A$ from $E$ to $F$.
\end{claim}
\begin{claimproof}
This is clear, as $a_1$ and $a_2$ lie in distinct $F$-classes in $A$.
\end{claimproof}

\begin{claim}
Suppose $d(a_1, a_2) \in (E, F]$. On $A$, define ${<^*_E} = {<_E \cup \arr}$. Then
\begin{enumerate}
\item $a_1 \arr a_2$ and $a_2 \arr a_1$ cannot both hold.
\item $E$ is a $<^*_E$-congruence.
\end{enumerate}
\end{claim}
\begin{claimproof}
\hspace{1 cm} \newline
$(1)$ Suppose $a_1 \arr a_2 \arr a_1$. Then there exist $x_1, x_2$ such that $a_1 \preceq_E x_1 \preceq_E a_2$, and $a_2 \preceq_E x_2 \preceq_E a_1$.

In particular, $x_1 \preceq_E x_2 \preceq_E x_1$, so $d(x_1, x_2) \leq E$. As $d(a_1, a_2) \not \leq E$, we may suppose $d(a_1, x_2) \not \leq E$. 

But $x_2 \preceq_E a_1$, so $x_2 <_E a_1 \preceq_E x_1$. Thus $x_2 <_E x_1$, which contradicts $x_2 \preceq_E x_1$.

$(2)$ We check that $E$ is a $<^*_E$-congruence. Since $d(a_1, a_2) \not \leq E$, it suffices without loss of generality to consider some $x \in A^*_0$ such that $d(a_1, x) \leq E$, $d(a_2, x) \in (E, F]$.

In this case, we claim
$$a_1 \arr a_2 \Longleftrightarrow x <_E a_2 \hspace{1 cm} a_2 \arr a_1 \Longleftrightarrow a_2 <_E x$$

The implications from right to left hold by the definition of $\arr$.

For the implication from left to right, we consider only the case $a_1 \arr a_2$, since the other is similar. By definition, there exists some $y$ such that $a_1 \preceq_E y \preceq_E a_2$. Then $x \preceq_E a_1 \preceq_E y$, so $x \preceq_E y$. Since $y \preceq_E a_2$, then $x \preceq_E a_2$. Since $d(x, a_2) \not\leq E$, we have $x <_E a_2$.
\end{claimproof}

Claims 63 and 64 dispose of the cases in which $d(a_1, a_2) \not\in (E, F]$. By Claim 65, if $d(a_1, a_2) \in (E, F]$ and $a_1 \arr a_2$, we may complete amalgam by determining $a_1 <_E a_2$, and vice versa if $a_2 \arr a_1$. If $d(a_1, a_2) \in (E, F]$ and neither $a_1 \arr a_2$ nor $a_2 \arr a_1$, we may complete the amalgam by arbitrarily determining either $a_1 <_E a_2$ or $a_2 <_E a_1$.
\end{proof}


\subsection*{Acknowledgements}
I am grateful to Gregory Cherlin for our many discussions concerning the material of this paper. I would also like to thank the anonymous reviewer for various helpful comments.


\end{document}